\documentclass[11pt]{amsart}
\usepackage{amssymb, amsmath}

\usepackage{anysize}

\marginsize{2.5cm}{2.5cm}{2.5cm}{2.5cm}
\theoremstyle{plain}

\newtheorem{proposition}{Proposition}

\theoremstyle{remark}
\newtheorem{remark}{Remark}

\newcommand{\Q}{\mathbb{Q}}
\newcommand{\Z}{\mathbb{Z}}

\title{More on Diophantine sextuples}



\author{Andrej Dujella}
\address{Department of Mathematics, University of Zagreb, Bijeni\v{c}ka cesta 30, 10000 Zagreb, Croatia}
\email{duje@math.hr}
\author{Matija Kazalicki}
\address{Department of Mathematics, University of Zagreb, Bijeni\v{c}ka cesta 30, 10000 Zagreb, Croatia}
\email{matija.kazalicki@math.hr}

\begin{document}

\maketitle

\section{Introduction}

A Diophantine $m$-tuple is a set of $m$ positive
integers with the property that the product of any two of its distinct
elements is one less then a square. If a set of nonzero rationals
has the same property, then it is called
a rational Diophantine $m$-tuple.
Diophantus of Alexandria found the first example of a rational Diophantine quadruple
$\{1/16, 33/16, 17/4, 105/16\}$, while the first Diophantine quadruple in integers was
found by Fermat, and it was the set $\{1,3,8,120\}$.
It is well-known that there exist infinitely many integer Diophantine quadruples
(e.g. $\{k,k+2,4k+4,16k^3+48k^2+44k+12\}$ for $k\geq 1$),
while it was proved in \cite{D-crelle} that an integer Diophantine sextuple does not exist and that there are
only finitely many such quintuples. A folklore conjecture is that there does not exist an integer
Diophantine quintuple. There is an even stronger conjecture which predicts that all integer
Diophantine quadruples $\{a,b,c,d\}$ satisfy the equation $(a+b-c-d)^2=4(ab+1)(cd+1)$ (such
quadruples are called regular).
However, in the rational case, there exist larger sets with the same property.
Euler found infinitely many rational Diophantine quintuples,
e.g. he was able to extend the Fermat quadruple to the rational quintuple
$\{ 1, 3, 8, 120, 777480/8288641\}$. Gibbs \cite{Gibbs} found the first rational Diophantine sextuple
$$\{ 11/192, 35/192, 155/27, 512/27, 1235/48, 180873/16\}, $$
while Dujella, Kazalicki, Miki\'c and Szikszai \cite{DKMS} recently proved that there exist
infinitely many rational Diophantine sextuples.
No example of a rational Diophantine septuple is known.
Moreover, we do not know any rational Diophantine quadruple
which can be extended to two different rational Diophantine sextuples.
On the other hand, by the construction from \cite{DKMS}, we know that there
exist infinitely many rational Diophantine triples,
each of which can be extended to rational Diophantine sextuples in infinitely
many ways.
In particular, there are infinitely many rational Diophantine sextuples
containing the triples
$\{15/14, -16/21, 7/6\}$ and $\{ 3780/73, 26645/252, 7/13140\}$.
The construction from \cite{DKMS} uses elliptic curves induced by Diophantine triples,
i.e. curves of the form $y^2 = (x+ab)(x+ac)(x+bc)$ where $\{a,b,c\}$ is a
rational Diophantine triple, with torsion group $\mathbb{Z}/2\mathbb{Z} \times
\mathbb{Z}/6\mathbb{Z}$ over $\mathbb{Q}$.

Piezas \cite{TP} studied Gibbs's examples of rational Diophantine sextuples which do not fit into the
construction from \cite{DKMS} and realized that most of them follow a common pattern:
they contain two regular subquadruples with two common elements (see Proposition \ref{prop:1}).
By studying sextuples of that special form, he obtained new simpler  parametric formulas
for rational Diophantine sextuples, and also obtained infinitely many sextuples
$\{a,b,c,d,e,f\}$ with fixed products $ab$ and $cd$ (e.g. $ab=24$ and $cd=9/16$).

In this paper, we will reformulate results from \cite{TP} in terms of the geometry of a certain algebraic variety parameterizing rational Diophantine quadruples, in fact the fiber product of three Edwards curves over $\mathbb{Q}(t)$, and obtain a method for generating (new) parametric formulas for rational Diophantine sextuples.

\section{Construction}
\subsection{Correspondence}

Let $\{a,b,c,d\}$ be a rational Diophantine quadruple with elements in $\Q$ or $\Q(t)$, and let
\begin{align*}
ab+1=t_{12}^2 \quad ac+1&=t_{13}^2 \quad ad+1=t_{14}^2\\
bc+1=t_{23}^2 \quad bd+1&=t_{24}^2 \quad cd+1=t_{34}^2.
\end{align*}

It follows that $(t_{12},t_{34},t_{13},t_{24},t_{14},t_{23}, m'=abcd)$ defines a point on an  algebraic variety $\mathcal{C}$  defined by the following equations:
\begin{align*}
(t_{12}^2-1)(t_{34}^2-1)&=m'\\
(t_{13}^2-1)(t_{24}^2-1)&=m'\\
(t_{14}^2-1)(t_{23}^2-1)&=m'.
\end{align*}

Conversely, the points $(\pm t_{12},\pm t_{34},\pm t_{13},\pm t_{24},\pm t_{14},\pm t_{23}, m')$ on $\mathcal{C}$ determine two rational Diophantine quadruples $\pm(a,b,c,d)$ (for example $a^2=(t_{12}^2-1)(t_{13}^2-1)/(t_{23}^2-1)$) provided that the elements $a,b,c$ and $d$ are rational, distinct and non-zero.
Note that if one element is rational, then all the elements are rational.

The projection $(t_{12},t_{34},t_{13},t_{24},t_{14},t_{23}, m') \mapsto m'$ defines a fibration of $\mathcal{C}$ over the projective line, and a generic fiber is the product of three curves $\mathcal{D}: (x^2-1)(y^2-1)=m'$. Any point on $\mathcal{C}$ corresponds to the three points $Q_1=(t_{12},t_{34})$, $Q_2=(t_{13},t_{24})$ and $Q_3=(t_{14}, t_{23})$ on $\mathcal{D}$. The elements of the quadruple corresponding to these three points are distinct if and only if no two of these points can be transformed from one to another by changing signs and switching coordinates, e.g. for the points $(t_{12}, t_{34})$, $(-t_{34}, t_{12})$ and $(t_{14},t_{23})$, we have that $a=d$.
\subsection{Extending quadruples to sextuples}

The following proposition gives a criterion for extending quadruples to sextuples.

\begin{proposition}[T. Piezas \cite{TP}]\label{prop:1}
Let $\{a, b, c, d\}$ be a rational Diophantine quadruple, and $x_1$ and $x_2$ the roots of
\[
(abcdx+2abc+a+b+c-d-x)^2=4(ab+1)(ac+1)(bc+1)(dx+1).
\]
If $x_1x_2 \ne 0$ and
\begin{equation}\label{eq:1}
(abcd-3)^2=4(ab+cd+3),
\end{equation}
then $\{a,b,c,d,x_1,x_2\}$ is a Diophantine sextuple.  Furthermore,
\begin{eqnarray*}
(a+b-x_1-x_2)^2 &=& 4(ab+1)(x_1x_2+1)\\
(c+d-x_1-x_2)^2 &=& 4(cd+1)(x_1x_2+1).
\end{eqnarray*}

\end{proposition}

Note that $x_1$ and $x_2$ coincide with the extensions of rational Diophantine quadruples
given in \cite[Theorem 1]{D-acta2}, and the condition (\ref{eq:1}) implies
that $x_1x_2+1=\left(\frac{a+b-c-d}{abcd-1}\right)^2$.

In this section, we will reformulate Proposition \ref{prop:1} in terms of the geometry of the algebraic variety $\mathcal{C}$.

The condition \eqref{eq:1} is equivalent to $t_{12} t_{34}= \pm t_{12} \pm t_{34}$, or $t_{34}=\pm t_{12}/(t_{12}\pm 1)$. For the rest of the paper, we set $t_{12} =t$, $t_{34}=t/(t-1)$ and $m'=(t^2-1)(\frac{t^2}{(t-1)^2}-1)=\frac{2t^2 + t - 1}{t - 1}$, and thus  condition \eqref{eq:1} is satisfied.

The curve $\mathcal{D}$ over $\Q(t)$
$$\mathcal{D}: (x^2-1)(y^2-1)=\frac{2t^2 + t - 1}{t - 1}$$ is birationally equivalent to the elliptic curve
$$E: S^2 = T^3 -2\cdot \frac{2t^2 - t + 1}{t-1}T^2 + \frac{(2t - 1)^2 (t + 1)^2}{(t-1)^2} T.$$
The map is given by $T = 2(x^2-1)y+2x^2-(2-m')$,
and $S = 2Tx$, where $m'=\frac{2t^2 + t - 1}{t - 1}$.

Denote by $P=\displaystyle\left[\frac{(2t-1)^2(t+1)}{t-1}, \frac{2t(2t-1)^2(t+1)}{t-1}\right]\in E(\Q(t))$ a point of infinite order on $E$, and by $R
 = \displaystyle\left[\frac{(t+1)(2t-1)}{(t-1)}, \frac{2(t+1)(2t-1)}{t-1}\right]$ a point of order $4$. The point $(t_{12},t_{34}) \in \mathcal{D}(\Q(t))$ corresponds to the point $P\in E(\Q(t))$.

\begin{proposition}\label{prop:MW}
The Mordell-Weil group of $E(\Q(t))$ is generated by $P$ and $R$.
\end{proposition}
\begin{proof}
It is enough to prove that the specialization homomorphism at $t_0=6$ is injective. Then one can easily check that the specializations of points $P$ and $R$ generate the Mordell-Weil group of $E_{t_0}(\Q)$.

We use the injectivity criterion from Theorem 1.3 in \cite{GT}. It states that given an elliptic curve $y^2=x^3+A(t)x^2+B(t)x$, where $A,B \in \mathbb{Z}[t]$, with exactly one nontrivial $2$-torsion point over $\Q(t)$, the specialization homomorphism at $t_0\in \Q$ is injective if the following condition is satisfied: for every nonconstant square-free divisor $h(t)\in \mathbb{Z}[t]$ of $B(t)$ or $A(t)^2-4B(t)$ the rational number $h(t_0)$ is not a square in $\Q$.

The claim follows (after clearing out the denominators in the defining equation of $E$).
\end{proof}

If $Q\in E$ is the point that corresponds to the point $(x,y)\in \mathcal{D}$, then the points $-Q$ and $Q+R$ correspond to the points $(-x,y)$ and $(y,-x)$.
Hence the triple $(Q_1,Q_2,Q_3)\in E(\Q(t))^3$ corresponds to the quadruple whose elements are not distinct if and only if there are two points, say $Q_i$ and $Q_j$, such that $Q_i=\pm Q_j+kR$, where $k \in \{0,1,2,3\}$.

If instead of $m'$, we fix on $\mathcal{C}$ coordinates $t_{12}, t_{13}, t_{23}$ we will obtain an elliptic curve on $\mathcal{C}$ consisting of points $(t_{34}, t_{24}, t_{14}, m')$ which satisfy
\begin{align*}
(t_{34}^2-1)&=\frac{m'}{(t_{12}^2-1)}\\
(t_{24}^2-1)&=\frac{m'}{(t_{13}^2-1)}\\
(t_{14}^2-1)&=\frac{m'}{(t_{23}^2-1)}.
\end{align*}

Thus, to the point $(t_{12},t_{34},t_{13},t_{24},t_{14},t_{23}, m')$ on $\mathcal{C}$ that corresponds to the rational quadruple $\{ a, b, c, d \}$, we associate the elliptic curve $E_{abc}:y^2=(x+ab)(x+ac)(x+bc)$ together with the point $W=[abcd,abc \cdot t_{14} t_{24} t_{34}]$.
A short calculation shows that if we denote by $V=[1, t_{12} t_{13} t_{23}]$ a point on $E_{abc}$, then $x_1$ and $x_2$ from Proposition \ref{prop:1} are given by
\[
x_1 = \frac{x(W+V)}{abc} \quad\textrm{ and }\quad x_2=\frac{x(W-V)}{abc}.
\]
For more details on using the elliptic curve $E_{abc}$ for extending rational Diophantine triples and quadruples see \cite[Theorem 1]{D-acta2}, \cite[Theorem 1]{D-Nom} and \cite[Proposition 2.1]{DKMS}.

\subsection{Degenerate case}
In this subsection we fix $Q_1=P$ and investigate conditions under which the point $(Q_1, Q_2, Q_3)\in E(\Q(t))\times E(\Q(t)) \times E(\Q(t))$ corresponds to the degenerate Diophantine sextuple (i.e. $x_1 x_2 =0$). We call such triple degenerate. Following the notation from the previous section, we see that the triple is degenerate if and only if $\pm W \pm V = [0,abc]\in E_{abc}(\Q(t))$ for some choice of the signs.

\begin{proposition} \label{prop:2}
Let $Q_2, Q_3 \in E(\Q(t))$. The triple $(Q_1, Q_2, Q_3)\in E(\Q(t))\times E(\Q(t)) \times E(\Q(t))$ is degenerate if and only if $\pm Q_1\pm Q_2 \pm Q_3= R$ for some choice of the signs.
\end{proposition}
\begin{proof}
Let $r=x(Q_2)$ and $s=x(Q_3)$. Direct calculation shows that the constant term of the polynomial from Proposition \ref{prop:1} is zero if and only if $g(r,s)h(r,s)=0$ where
\begin{eqnarray*}
g(r,s)&= \left( (-1+t)^2 r s-(1+t)^2(-1+2t)(r+s)+(1+t)^2(-1+2t)^2 \right)^2 - 16 r s t^2(1+t)^2(-1+2t),\\
h(r,s) &= \left( (-1+t)^2 r s-(1-t)^2(-1+2t)(r+s)+(1+t)^2(-1+2t)^2 \right)^2 - 16 r s t^2(1-t)^2(-1+2t).
\end{eqnarray*}

One can check that $r$ and $s$ satisfy this equation if $\pm Q_1\pm Q_2 \pm Q_3= R$ for some choice of the signs.

Conversely, both $g(r,s)=0$ and $h(r,s)=0$ define a curve that is birationally equivalent to $E$. Hence, we have a degree four map from ``the degeneracy locus'' in $E\times E$ to $E$ given by $(Q_2,Q_3) \mapsto (x(Q_2), x(Q_3))$. Since we already have $8$ irreducible components in ``the degeneracy locus'' (one for the each choice of the signs), the claim follows.

\end{proof}

\subsection{Rationality}

Given a triple $(Q_1, Q_2, Q_3)\in E(\Q(t))\times E(\Q(t)) \times E(\Q(t))$, where $Q_1=P$, we want to know if the corresponding Diophantine quadruple is rational. It is enough to prove that one element is rational.

A short calculation shows that for the point $(S,T)\in E(\Q(t))$ we have
\begin{equation}\label{eq:2}
x^2-1=\left(\frac{S}{2T}\right)^2-1=T\left(\frac{T - \frac{2t^2 + t - 1}{t - 1}}{2T}\right)^2=:f(T).
\end{equation}
Since
\[
\begin{array}{ll}
a^2 &=\frac{f(Q_1)f(Q_2)f(Q_3)}{m'} \equiv x(Q_1)x(Q_2)x(Q_3)m'\equiv (2t-1)x(Q_2)x(Q_3)\\ & \equiv x(-P+R)x(Q_2)x(Q_3) \pmod{\Q(t)^{\times 2}}
\end{array}
\]
 for the rationality of $a$ it is enough to prove that $x(-P+R)x(Q_2)x(Q_3)$ is a square in $\Q(t)$.

Since the point $(0,0)\in E(\Q(t))$ is a point of order $2$, the usual 2-descent homomorphism  $E(\Q(t)) \rightarrow \Q(t)^\times/\Q(t)^{\times 2}$, which is for non-torsion points defined by $(T,S) \mapsto T$ (note that $(0,0)\mapsto 1$), implies the following proposition.

\begin{proposition}\label{prop:3} Let $Q_2, Q_3 \in E(\Q(t))$.
\begin{enumerate}
	\item [a)] If $Q_2+Q_3\equiv \mathcal{O}  \bmod{2E(\Q(t))}$ then $a^2 \equiv (2t-1) \bmod{\Q(t)^{\times 2}}$.
	\item [b)] If $Q_2+Q_3\equiv R \bmod{2E(\Q(t))}$ then $a^2 \equiv (t-1)(t+1) \bmod{\Q(t)^{\times 2}}$.
	\item [c)] If $Q_2+Q_3\equiv P \bmod{2E(\Q(t))}$ then $a^2 \equiv (t-1)(t+1)(2t-1) \bmod{\Q(t)^{\times 2}}$.
	\item [d)] If $Q_2+Q_3\equiv P+R \bmod{2E(\Q(t))}$ then $a^2 \equiv 1 \bmod{\Q(t)^{\times 2}}$.
\end{enumerate}
\end{proposition}
\begin{remark}
In the cases a) and b) we can still obtain parametric families of Diophantine sextuples if we specialize to those $t's$ for which $2t-1$ and $(t-1)(t+1)$ are squares (e.g. if we specialize $t$ to $\frac{t^2+1}{2}$ and $\frac{t^2+1}{2t}$). Concerning the case c), the elliptic curve $y^2=(x-1)(x+1)(2x-1)$ has Mordell-Weil group isomorphic to $\Z/2\Z+\Z/4\Z$.
\end{remark}
\begin{remark}
The proposition covers all the possibilities, since the Mordell-Weil group of $E(\Q(t))$ is generated by $P$ and $R$ (see Proposition \ref{prop:MW}).
\end{remark}

\section{Examples}
\subsection{Family corresponding to $(P,2P,4P)$}
For an illustration, we calculate a parametric family $\{a, b, c, d, e, f \}$ of rational Diophantine sextuples that corresponds to the triple $(P,2P,4P)$. It follows from Proposition \ref{prop:2} that the triple is not degenerate. The rationality of the sextuple will follow if we replace $t$ by $\frac{t^2+1}{2}$ (see part a) of Proposition \ref{prop:3}). Then, the corresponding Diophantine quadruple is equal to
\begin{eqnarray*}
a &\!\!=\!\!& \frac{ (t^{8} - 8 t^{6} - 14 t^{4} + 32 t^{2} - 27) \cdot (t^{8} + 26 t^{4} - 40 t^{2} - 3)}{64\cdot(t - 1) \cdot t \cdot (t + 1) \cdot (t^{4} - 2 t^{2} + 5) \cdot (t^{4} + 6 t^{2} - 3)}, \\
b &\!\!=\!\!& \frac{16 \cdot t \cdot (t - 1)^{2} \cdot (t + 1)^{2} \cdot (t^{2} + 3) \cdot (t^{4} - 2 t^{2} + 5) \cdot (t^{4} + 6 t^{2} - 3)}{(t^{8} - 8 t^{6} - 14 t^{4} + 32 t^{2} - 27) \cdot (t^{8} + 26 t^{4} - 40 t^{2} - 3)}, \\
c &\!\!=\!\!& \frac{t \cdot (t^{8} - 8 t^{6} - 14 t^{4} + 32 t^{2} - 27) \cdot (t^{8} + 26 t^{4} - 40 t^{2} - 3)}{(t - 1) \cdot (t + 1) \cdot (t^{2} - 3)^{2} \cdot (t^{2} + 1)^{2} \cdot (t^{4} - 2 t^{2} + 5) \cdot (t^{4} + 6 t^{2} - 3)},\\
d &\!\!=\!\!& \frac{4 \cdot t \cdot (t^{2} - 3)^{2} \cdot (t^{2} + 1)^{2} \cdot (t^{4} - 2 t^{2} + 5) \cdot (t^{4} + 6 t^{2} - 3)}{(t - 1) \cdot (t + 1) \cdot (t^{8} - 8 t^{6} - 14 t^{4} + 32 t^{2} - 27) \cdot (t^{8} + 26 t^{4} - 40 t^{2} - 3)}.
\end{eqnarray*}
Using Proposition \ref{prop:1} (let $e=x_1$ and $f=x_2$), we find that $e = e_1/e_2$ and $f = f_1/f_2$ are equal to
{\small
\begin{eqnarray*}
e_1&\!\!=\!\!&  (t + 1) \cdot (t^{2} - 2 t + 3) \cdot (t^{2} + 2 t - 1) \cdot (t^{6} - 2 t^{5} + t^{4} + 12 t^{3} + 7 t^{2} - 2 t - 9) \\ & & \mbox{} \cdot (t^{6} + 2 t^{5} - 3 t^{4} + 4 t^{3} - 17 t^{2} + 18 t + 3) \\ & & \mbox{} \cdot (t^{12} - 4 t^{11} + 6 t^{10} + 20 t^{9} - t^{8} + 24 t^{7} - 12 t^{6} - 88 t^{5} - 177 t^{4} + 364 t^{3} - 90 t^{2} - 60 t + 81)\\ & & \mbox{} \cdot (t^{12} + 4 t^{11} - 2 t^{10} - 4 t^{9} - 41 t^{8} + 40 t^{7} + 100 t^{6} - 72 t^{5} + 63 t^{4} + 212 t^{3} - 66 t^{2} - 180 t + 9), \\
e_2&\!\!=\!\!&
64 \cdot (t - 1) \cdot t \cdot (t^{2} - 3)^{2} \cdot (t^{2} + 1)^{4} \cdot (t^{4} - 2 t^{2} + 5) \cdot (t^{4} + 6 t^{2} - 3) \cdot (t^{8} - 8 t^{6} - 14 t^{4} + 32 t^{2} - 27) \\ & & \mbox{} \cdot (t^{8} + 26 t^{4} - 40 t^{2} - 3), \\
f_1&\!\!=\!\!&
  (t - 1) \cdot (t^{2} - 2 t - 1) \cdot (t^{2} + 2 t + 3) \cdot (t^{6} - 2 t^{5} - 3 t^{4} - 4 t^{3} - 17 t^{2} - 18 t + 3) \cdot (t^{6} + 2 t^{5} + t^{4} - 12 t^{3} + 7 t^{2} + 2 t - 9) \\ & & \mbox{} \cdot (t^{12} - 4 t^{11} - 2 t^{10} + 4 t^{9} - 41 t^{8} - 40 t^{7} + 100 t^{6} + 72 t^{5} + 63 t^{4} - 212 t^{3} - 66 t^{2} + 180 t + 9)\\ & & \mbox{} \cdot (t^{12} + 4 t^{11} + 6 t^{10} - 20 t^{9} - t^{8} - 24 t^{7} - 12 t^{6} + 88 t^{5} - 177 t^{4} - 364 t^{3} - 90 t^{2} + 60 t + 81), \\
f_2&\!\!=\!\!&
64 \cdot t \cdot (t + 1) \cdot (t^{2} - 3)^{2} \cdot (t^{2} + 1)^{4} \cdot (t^{4} - 2 t^{2} + 5) \cdot (t^{4} + 6 t^{2} - 3) \cdot (t^{8} - 8 t^{6} - 14 t^{4} + 32 t^{2} - 27)\\ & & \mbox{} \cdot (t^{8} + 26 t^{4} - 40 t^{2} - 3).\\
\end{eqnarray*} }

\subsection{Rank two examples}

If we specialize $t$ to $t^2+1$, the elliptic curve $E$ will have another point of infinite order (independent of $P$), $S=\left[\frac{(2+t^2)^2}{t^2}, \frac{(2+t^2)^2(1+t^2)}{t^3}\right]$.
Now the triple $(P,2P+S,R+S-P)$ is not degenerate and satisfies the condition
of Proposition \ref{prop:3}(d). Our construction gives the following family of
rational Diophantine sextuples

\begin{eqnarray*}
a&\!\!=\!\!&  \frac{(t^3+3t^2+t+1)\cdot (t^3+t^2+3t+1)\cdot (2t-1)}{2 \cdot(t-2)\cdot (t-1)\cdot (t^2+t+1) \cdot (t+1)}, \\
b&\!\!=\!\!& \frac{2\cdot (t-1)\cdot (t^2+t+1)\cdot (t+1)\cdot (t^2+2)\cdot (t-2)\cdot t^2}{(t^3+3 t^2+t+1) \cdot(t^3+t^2+3 t+1)\cdot (2t-1)},\\
c&\!\!=\!\!&  \frac{(t^3+3t^2+t+1)\cdot (t^3+t^2+3t+1)\cdot (t-2)}{2\cdot (2t-1)\cdot (t-1)\cdot (t^2+t+1) \cdot(t+1)\cdot t^2},\\
d&\!\!=\!\!& \frac{2\cdot (2t^2+1)\cdot (2t-1)\cdot (t-1)\cdot (t^2+t+1)\cdot (t+1)}{t^2\cdot (t^3+3t^2+t+1) \cdot(t^3+t^2+3t+1)\cdot (t-2)},\\
e&\!\!=\!\!& \frac{8 \cdot t^2\cdot (t-1)\cdot (2 t+1)\cdot (t+2)\cdot (t+1)\cdot (t^2+1)}{(t-2)\cdot (2t-1)\cdot (t^2+t+1)\cdot (t^3+t^2+3t+1)\cdot (t^3+3t^2+t+1)},\\
f&\!\!=\!\!& \frac{3\cdot (3 t^2+2t+1)\cdot (t^2+2 t+3)\cdot (t^4+1)\cdot (t^4+4 t^2+1)}{2\cdot (t-1)\cdot (t-2)\cdot (2t-1)\cdot (t+1)\cdot (t^2+t+1)\cdot (t^3+t^2+3t+1)\cdot (t^3+3 t^2+t+1)}.\\
\end{eqnarray*} 

If we further require $2(t^2+1)$ to be a square, then the resulting parametrization $t \mapsto 1+\left(\frac{4t^2 - 8t - 4}{4t^2 + 8t - 4}\right)^2$ yields a point $K$ on $E$
$$K=\left[\frac{(t^{2} - 2t + 1) \cdot (t^{2} + 2 t + 3) \cdot (t^{2} + 2t+ 1)^{2}}{(t^{2} - 2 t - 1)^{2} \cdot (t^{2} + 2 t - 1)^{2}},\frac{4 (t^{2} - 2 t + 1) \cdot (t^{2} + 1) \cdot (t^{2} + 2 t + 3) \cdot (t^{2} + 2 t + 1)^{2}}{(t^{2} + 2 t - 1)^{2} \cdot (t^{2} - 2 t - 1)^{3}}\right],$$
with the property that $2K=S$.
When we apply our construction to the triple $(P, K, -2K+R)$, we obtain a very simple family of sextuples also found by Piezas \cite{TP}

\begin{eqnarray*}
a&\!\!=\!\!&  \frac{(t^2-2t-1)\cdot(t^2+2t+3)\cdot(3t^2-2t+1)}{4t\cdot(t^2-1)\cdot(t^2+2t-1)},\\
b&\!\!=\!\!& \frac{4t\cdot(t^2-1)\cdot(t^2-2t-1)}{(t^2+2t-1)^3},\\
c&\!\!=\!\!& \frac{4t\cdot(t^2-1)\cdot(t^2+2t-1)}{(t^2-2t-1)^3}, \\
d&\!\!=\!\!& \frac{(t^2+2t-1)\cdot(t^2-2t+3)\cdot(3t^2+2t+1)}{4t\cdot(t^2-1)\cdot(t^2-2t-1)},\\
e&\!\!=\!\!& \frac{ -t\cdot(t^2+4t+1)\cdot(t^2-4t+1)}{(t-1)\cdot(t+1)\cdot(t^2+2t-1)\cdot(t^2-2t-1)},\\
f&\!\!=\!\!& \frac{(t-1)\cdot(t+1)\cdot(3t^2-1)\cdot(t^2-3)}{4t\cdot(t^2+2t-1)\cdot(t^2-2t-1)}.\\
\end{eqnarray*} 

{\bf Acknowledgements:} {Authors acknowledge support from the QuantiXLie Center of Excellence.
A.D. was supported by the Croatian Science Foundation under the project no. 6422.}

\end{document}